\documentclass[12pt]{amsart}
\usepackage{amssymb}
\usepackage{graphicx}

\newcommand{\N}{{\mathbb N}}
\newcommand{\C}{{\mathbb C}}

\newcommand{\tef}{transcendental entire function}

\theoremstyle{plain}
\newtheorem{theorem}{Theorem}[section]
\newtheorem{corollary}[theorem]{Corollary}
\newtheorem{proposition}[theorem]{Proposition}
\newtheorem{lemma}[theorem]{Lemma}
\theoremstyle{definition}

\theoremstyle{remark}

\theoremstyle{problem}

\theoremstyle{example}

\newcommand{\dist}{\textrm{dist\,}}
\newcommand{\ind}{\textrm{ind\,}}
\newcommand{\ov}{\overline}

\newcommand{\rnN}{r_{n+N}}
\newcommand{\AnN}{A_{n+N}}
\newcommand{\rnu}{r_{n+1}}
\newcommand{\CnN}{C_{n+N}}

\newcommand{\dnu}{\ell_{\gamma,n}}
\newcommand{\dnd}{\ell_{C,n}}

\title[Permutable entire functions]{Permutable entire functions and multiply connected wandering domains}

\vspace{5cm}

\author{Anna Miriam Benini}
\address{Dipartimento di Matematica\\
Universit\`a di Roma `Tor Vergata'\\
v. ricerca scientifica 1\\
Roma\\
Italy}
\email{ambenini@gmail.com}

\author{Philip J. Rippon}
\address{Department of Mathematics and Statistics \\
The Open University \\
   Walton Hall\\
   Milton Keynes MK7 6AA\\
   UK}
\email{Phil.Rippon@open.ac.uk}

\author{Gwyneth M. Stallard}
\address{Department of Mathematics and Statistics \\
The Open University \\
   Walton Hall\\
   Milton Keynes MK7 6AA\\
   UK}
\email{Gwyneth.Stallard@open.ac.uk}

\thanks{The first author is  supported by the ERC grant HEVO: Holomorphic Evolution Equations, n.~277691, and the last two authors by the EPSRC grant EP/K031163/1.}


\begin{document}

\maketitle
\begin{abstract} Let~$f$ and~$g$ be permutable transcendental entire functions. We use a recent analysis of the dynamical behaviour in multiply connected wandering domains to make progress on the long standing conjecture that the Julia sets of~$f$ and~$g$ are equal; in particular, we show that $J(f)=J(g)$ provided that neither~$f$ nor~$g$ has a simply connected wandering domain in the fast escaping set.
\end{abstract}

\section{Introduction}
\setcounter{equation}{0}
Let~$f$ be a rational map or a {\tef}. The \emph{Fatou set} $F(f)$ is defined as the set of points in a neighborhood of which the family of iterates of~$f$ forms a normal family, while the \emph{Julia set} is its complement and can be characterized (see \cite{Ba3}) as the closure of the repelling periodic points of~$f$. Connected components of the Fatou set, known as \emph{Fatou components}, are either periodic, preperiodic or wandering. Wandering components of the Fatou set are called \emph{wandering domains}, and do not occur for rational functions \cite{Su}. For both transcendental entire functions and rational functions, periodic and preperiodic components can be completely classified according to the dynamics within them (see \cite{Mi},\cite{BerSur}).

Two holomorphic functions $f$ and~$g$ are said to be \emph{permutable} or \emph{commuting} if they satisfy the equation
\begin{equation}\label{Commuting}
f\circ g=g\circ f.
\end{equation}
It is natural to ask whether two functions satisfying the relation \eqref{Commuting} have the same dynamical behaviour and, in particular, whether they have the same Julia set and, equivalently, the same Fatou set.


It was shown over 90 years ago by both Fatou \cite{Fa1} and Julia \cite{Ju} that if two rational functions commute, then they have the same Julia set. Under some conditions the opposite implication is true; that is, in some cases, $J(f)=J(g)$ implies that $f\circ g=g\circ f$.
Another fascinating result is that, for rational functions, classes of commuting functions can be completely characterized (see  \cite{Ri}, \cite{Er} respectively for a topological-algebraic and a dynamical approach  to the classification of classes of commuting functions).

For transcendental entire functions, the question of characterizing classes of commuting functions seems currently out of reach, though Baker \cite[Theorem~1]{Ba1} showed that if $f$ is a transcendental entire function, then the set of entire functions that commute with $f$ is at most countable. However, the conjecture that commuting implies having the same Julia set, although not solved yet, has seen significant progress.

For rational functions the strategy used by Julia \cite[page~143]{Ju} to prove that $J(f)=J(g)$ whenever $f\circ g=g\circ f$ was to note that if $\alpha$ is a repelling periodic point of~$f$, then the sequence $(g^n(\alpha))$ consists of repelling periodic points of $f$ and must (since~$f$ is rational) include a periodic point of~$g$, which is shown to be repelling. So $\alpha \in J(g)$ and hence $J(f)\subset J(g)$, by the density of repelling periodic points in the Julia set. Thus $J(f)=J(g)$, since the roles of $f$ and $g$ can be interchanged.

Fatou's proof of this fact \cite[page~365]{Fa1} involved showing that $g(F(f))\subset F(f)$ or, equivalently, that $g^{-1}(J(f))\subset J(f)$; this implies that the iterates $(g^n)$ form a normal family in $F(f)$, by Montel's theorem, so $F(f) \subset F(g)$ and hence $J(g)\subset J(f)$.

Fatou's appproach appears more promising than Julia's when~$f$ and~$g$ are {\tef}s. However, a difficulty of proving that $g(F(f))\subset F(f)$ in the transcendental case arises from the possible presence of Fatou components in which the iterates of~$f$ tend to the essential singularity at~$\infty$; such components can be \emph{Baker domains} (that is, periodic Fatou components in which the iterates tend to~$\infty$) or wandering domains in which the iterates tend to~$\infty$. Indeed, the following result was given by Baker.
\begin{theorem}[\cite{Ba2}, Lemma 4.3(i)]\label{Baker}
If $f$ and $g$ are permutable transcendental entire functions and in each component of~$F(f)$ the sequence $(f^n)$ has at least one convergent subsequence with a finite limit function, then $g(F(f))\subset F(f)$.
\end{theorem}
It was shown in \cite{EL} that any {\tef} with a bounded set of singular values cannot have Fatou components in which the iterates tend to infinity. Thus it follows from Theorem~\ref{Baker} that in this class (called the \emph{Eremenko-Lyubich class}) permutable functions always have the same Julia set.

Langley \cite{La} showed that if $f$ and~$g$ are permutable functions of finite order with no wandering domains, then $J(f)=J(g)$. At about the same time, Bergweiler and Hinkkanen \cite{BH} introduced the so-called \emph{fast escaping set} $A(f)$ (see Section \ref{MCWD} for a precise definition) and used it to prove a result that includes the following.

\begin{theorem}[\cite{BH}, Theorem~2]\label{BergHink}
If $f$ and $g$ are permutable transcendental entire functions such that $A(f)\subset J(f)$ and $A(g)\subset J(g)$, then $J(f)=J(g)$. In particular, this holds if $f$ and $g$ have no wandering domains.
\end{theorem}

It was shown in \cite{BH} that any Fatou component in $A(f)$ is a wandering domain, so in particular a Baker domain is never in the fast escaping set, and in \cite[Theorem~1.2]{RS2} that if a Fatou component meets $A(f)$, then it is contained in $A(f)$.

The proof of Theorem~\ref{BergHink} shows that if~$f$ and~$g$ are permutable transcendental entire functions and~$U$ is a Fatou component of~$f$ that is not in $A(f)$, then $g(U)\subset F(f)$; see Proposition~\ref{Not fast escaping} for the details of this argument. Therefore, in order to show that $J(f)=J(g)$ whenever $f$ and~$g$ commute, the only thing left to prove is that if~$f$  has a fast escaping Fatou component~$U$, then its image $g(U)$ is contained in the Fatou set of~$f$.

Fast escaping wandering domains can be either simply connected or multiply connected. Multiply connected wandering domains were first constructed by Baker \cite{Ba5} and they are \emph{always} fast escaping (see \cite{RS}), whereas most known simply connected wandering domains are not fast escaping. Indeed, the only known examples of functions with simply connected fast escaping wandering domains were given by Bergweiler \cite{Be2} and Sixsmith \cite{Si}. It is an open problem, mentioned in \cite{RS2}, whether simply connected wandering domains in $A(f)$ can be unbounded.

A detailed analysis of the dynamics in multiply connected wandering domains was recently given in \cite{BRS}, and is described in Section~\ref{MCWD}. This analysis and the arguments from \cite{BH} will be our main tools in proving the following result.

\begin{theorem}\label{main}
Let~$f$ and~$g$ be permutable transcendental entire functions such that~$f$ and~$g$ have no simply connected fast escaping wandering domains. Then $J(f)=J(g)$.
\end{theorem}
In Section~3 we state and prove a more precise result than Theorem~\ref{main}, and give an intriguing corollary of it.

%

\subsection*{Acknowledgments} The first author is grateful to Filippo Bracci for first introducing her to the problem of commuting functions.

\section{Background on multiply connected wandering domains}\label{MCWD}
\setcounter{equation}{0}
From now on we let~$f$ denote a {\tef} and let $U$ denote a multiply connected wandering domain of $f$, with forward iterates $U_n=f^n(U)$, $n\in\N$. Baker showed that multiply connected wandering domains have the following properties.
\begin{theorem}[\cite{Ba4} and  \cite{Ba2}, Theorem~3.2]\label{Baker1}
If $U$ is a multiply connected wandering domain, then
\begin{enumerate}
\item each $U_n$ is bounded and multiply connected;
\item for all sufficiently large $n$, $U_n$ lies in a bounded component of $\C\setminus U_{n+1}$;
\item $\dist(U_n,0)\to\infty$ as $n\to\infty$.
 \end{enumerate}
\end{theorem}

Let $M(r,f)$ denote the \emph{maximum modulus} of $f$; that is,
\[M(r,f):=\max_{|z=r|} |f(z)|.\]
Since $f$ is a transcendental entire function, we have
\begin{equation}\label{logMr}
\lim_{r\to\infty}\frac{\log M(r,f)}{\log r}=\infty.
\end{equation}
It was shown in \cite{RS} that any multiply connected wandering domain is contained in the \emph{fast escaping set} $A(f)$, defined as follows (see \cite{BH}):
\[A(f):=\{z\in\C:\text{ there exists $L\in\N$ with $|f^n(z)|>M(R,f^{n-L})$ for $n>L$}\}.\]
It was pointed out in \cite{BH} that the fast escaping set of a {\tef} is always non-empty, and that this follows from the construction of an escaping point by Eremenko~\cite{ErFast}. For more on the fast escaping set, including a detailed proof that it is non-empty, see \cite{RS2}.

We use several results from \cite{BRS} concerning multiply connected wandering domains. Here we use the following notation for an annulus:
\[A(r,R) = \{z\in\C : r < |z| < R\},\quad 0<r<R.\]

\begin{theorem}[\cite{BRS}, Theorem 1.2]\label{BRS1}
Let f be a transcendental entire function with a multiply connected wandering domain~$U$. For each $z_0\in U$ and each open set $V\subset U$ containing $z_0$, there exists
$\alpha > 0$ such that, for sufficiently large $n \in \N$,
\[U_n \supset f^n(V) \supset A({r_n}^{1-\alpha}, {r_n}^{1+\alpha})=:A_n,\]
where $r_n:=|f^n(z_0)|$.
\end{theorem}

By Theorem \ref{BRS1}, for any $z_0\in U$, there exists $\alpha>0$ such that, for sufficiently large~$n$, the maximal round annulus centered at~$0$, contained in $U_n$, and containing $f^n(z_0)$, is of the form  $B_n=A(r_n^{a_n},r_n^{b_n})$ with $0< a_n\leq 1-\alpha\leq
1+\alpha\leq b_n<\infty$, where $r_n=|f^n(z_0)|$. The union of these annuli $B_n$ forms an absorbing set for the dynamics of~$f$ on~$U$, in the following sense.

\begin{theorem}[\cite{BRS}, Theorem 1.3]\label{BRS2}
Let $f$ be a transcendental entire function with a multiply connected
wandering domain $U$ and let $z_0 \in U$. Then, for each compact subset $K$ of $U$, there exists $N\in \N$ such that
\begin{equation}
 f^n(K) \subset B_n:=A(r_n^{a_n},r_n^{b_n})\subset U_n, \text{ for } n \ge N,
\end{equation}
where $r_n=|f^n(z_0)|$ and $0< a_n\leq 1-\alpha\leq 1+\alpha\leq b_n<\infty$, with $\alpha$ as in Theorem~\ref{BRS1}.
\end{theorem}

We also use a special case of the following result.
\begin{lemma}[\cite{BRS}, Lemma 4.3 part (a)]\label{BRS3}
Let~$f$ be a transcendental entire function with a multiply connected wandering domain $U$. Let $z_0 \in U$ and, for $n \in \N$, let $\delta_n=1/\sqrt{\log r_n}$ with $r_n=|f^n(z_0)|$.
Then there exists $N \in \N$ such that, if we have
\begin{equation}\label{Un}
U_n\supset A(r_n^{1-2\pi \delta_n},r_n^{1+2\pi \delta_n}),
\end{equation}
for some $n \ge N$, then
\begin{equation}\label{rnm}
r_{n+m } \geq  M(r_n, f^m)^{1 -\delta_n}, \text{ for }m \in \N.
\end{equation}
\end{lemma}
Since $\delta_n\to 0$ as $n\to \infty$, by Theorem~\ref{Baker1}, the hypothesis \eqref{Un} holds for all sufficiently large $n$, by Theorem~\ref{BRS1}. Therefore, since $M(r_n, f^m)>1$, for all $m\in\N$ and sufficiently large~$n$, we have the following corollary of Lemma~\ref{BRS3}.
\begin{corollary}\label{BRS-Cor}
Let~$f$ be a transcendental entire function with a multiply connected wandering domain $U$, let $z_0 \in U$ and let $r_n=|f^n(z_0)|$ for $n\in \N$. Then, for all sufficiently large~$n$, we have
\begin{equation}\label{rnm}
r_{n+m } \geq  M(r_n, f^m)^{1/2}, \text{ for }m \in \N.
\end{equation}
\end{corollary}

\section{Proof of Theorem~\ref{main}}
\setcounter{equation}{0}
Theorem~\ref{main} is an immediate consequence of the following more precise result. For any {\tef} $f$ we denote by $F^*(f)$ the union of those components of $F(f)$ that are \emph{not} simply connected fast escaping wandering domains.

\begin{theorem}[Closely related Fatou sets]\label{main1}
Let~$f$ and~$g$ be permutable transcendental entire functions. Then
\[
F^*(f)\subset F(g)\quad \text{and}\quad F^*(g)\subset F(f).
\]
\end{theorem}
Under the hypotheses of Theorem~\ref{main}, we have $F^*(f)=F(f)$ and $F^*(g)=F(g)$, so we deduce from Theorem~\ref{main1} that $F(f)=F(g)$, as required.

The idea of the proof of Theorem~\ref{main1} is to show that, if~$f$ and~$g$ commute and~$U$ is any Fatou component of~$f$ other than a simply connected fast escaping component, then $g(U)\subset F^*(f)$. There are two cases to consider, the first being a version of Theorem~\ref{BergHink}.

\begin{proposition}[Not fast escaping]\label{Not fast escaping} Let $f$ and $g$ be permutable transcendental entire functions and let~$U$ be a Fatou component of $f$ that is not fast escaping. Then $g(U)\subset V$, where~$V$ is a Fatou component of $f$ that is not fast escaping.
\end{proposition}
\begin{proof}
Suppose for a contradiction that $z\in U$ but $g(z)\in J(f)$. In \cite[Lemma~3 and Theorem~5]{BH} it was shown that for any {\tef}~$f$ we have
\begin{equation}\label{JfAf}
J(f)\subset \ov{A(f)},
\end{equation}
and that for permutable {\tef}s $f$ and $g$ we have
\begin{equation}\label{backfast}
g^{-1}(A(f))\subset A(f).
\end{equation}
By \eqref{JfAf}, there exists a point~$z'$ in~$U$ such that $g(z')\in A(f)$. By \eqref{backfast}, however, we deduce that $z'\in A(f)$, which is a contradiction. Hence $g(z)\in F(f)$. Also $g(z)$ is not fast escaping, by \eqref{backfast}, as required.
\end{proof}

Our key new contribution to the commuting problem for {\tef}s is the following proposition.

\begin{proposition}[Multiply connected wandering domains]\label{MCWD in Fatou Set} Let $f$ and $g$ be permutable transcendental entire functions and let~$U$ be a multiply connected wandering domain of~$f$. Then $g(U)\subset V$, where $V$ is a multiply connected wandering domain  of~$f$.
\end{proposition}
Theorem~\ref{main1} follows immediately from Propositions~\ref{Not fast escaping} and~\ref{MCWD in Fatou Set}, since together they imply that $g(F^*(f))\subset F^*(f)$, so $F^*(f)\subset F(g)$ by Montel's theorem.

{\it Remark}\;\; Baker proved in \cite{Ba2} that if the transcendental entire functions~$f$ and~$g$ commute, then $g(J(f)\subset J(f)$. It follows that in both Propositions~\ref{Not fast escaping} and~\ref{MCWD in Fatou Set} we can replace the statement that $g(U)\subset V$ by the stronger statement that $g(U)=V$.

The following well-known result is needed in the proof of Proposition~\ref{MCWD in Fatou Set} (see \cite[Lemma 2.2]{Ba2} or \cite{BerSur}).
\begin{lemma}[Blowing up property]\label{bup}
Let  $f$ be a {\tef}  function and $V$ be a neighborhood of a point $\zeta\in J(f)$. For any compact set $K$ not containing an exceptional point of $f$ there exists $n_K$ such that $f^n(V)\supset K$ for all $n\geq n_K$.
\end{lemma}
An \emph{exceptional point} is one whose backward orbit is finite and, by Montel's theorem, a {\tef} has at most one exceptional point in~$\C$.

We also need the following result about the hyperbolic lengths of closed curves in annuli; see \cite[Section~12.2]{BM}, for example, for the special case when $r=1/R$ from which Lemma~\ref{ann-hyp} follows by the invariance of hyperbolic length under a conformal mapping.

\begin{lemma}[Closed curves in annuli]\label{ann-hyp}
Let $\gamma$ be a piecewise smooth closed curve in the annulus $A=A(r,R)$, where $0<r<R$, such that $\ind\,(\gamma,0)=n \neq 0$. Then the hyperbolic length $\ell(\gamma)$ of $\gamma$ in~$A$ satisfies
\[
\ell(\gamma)\ge \frac{2\pi^2 n}{\log R/r},
\]
with equality if and only if~$\gamma$ is a monotonic parametrization of the unit circle traversed~$n$ times.
\end{lemma}

\begin{proof}[Proof of Proposition~\ref{MCWD in Fatou Set}]
Let~$U$ be a multiply connected wandering domain of~$f$. We begin by showing that $g(U)\subset F(f)$. Suppose for a contradiction that there exists $z\in U$ with $g(z)\in J(f)$. Since holomorphic maps are open, $g(U)$ is an open neighborhood of $g(z)$. So, by Lemma~\ref{bup}, there exists a point $z_0\in U$ such that $f^N(g(z_0))\in U$ for some $N\geq 2$. Hence there exists a neighborhood $V\subset U$ of $z_0$ with
\begin{equation}\label{Eq1}
\ov{f^N(g(V))}\subset U.
\end{equation}
Let $r_n:=|f^n(z_0)|$ for $n\in\N$. By Theorem~\ref{Baker1} and Theorem~\ref{BRS1}, we have $r_n \to \infty$ as $n\to\infty$ and the sequence $(r_n)$ is eventually strictly increasing.

As before, let $U_n=f^n(U)$ for $n\in \N$. By Theorem~\ref{BRS1}, applied to $V$, there exists $\alpha>0$ and $N_1\in\N$ such that, for $n\geq N_1$,
\begin{equation}\label{Eq2}
U_{n+N}\supset f^{n+N}(V)\supset \AnN:= A(r_{n+N}^{1-\alpha},r_{n+N}^{1+\alpha}).
\end{equation}
Consider the circle $C_{n+N}=\{z\in\C: |z|=r_{n+N}\}$ and its image $\gamma_n$ under $g$. By \eqref{Eq2}, the commuting hypothesis and \eqref{Eq1}, applied in this order, we have
\begin{equation}\label{Eq3}
\gamma_n =g(C_{n+N})\subset g(\AnN)\subset g(f^{n+N}(V))= f^{n+N}(g(V))\subset f^n(U)=U_n.
\end{equation}
Recall that $r_{n+N}\to\infty$ as $n\to \infty$. We claim that, if $n$ is sufficiently large, then $\gamma_n=g(C_{n+N})$ must wind round the origin at least once. In fact, by continuity and the fact that $\gamma_n\subset U_n$, we have $\ind(\gamma_n,0)=\ind(\gamma_n,\zeta)$ for all $\zeta\ne 0$ sufficiently small. By Picard's Theorem,~$g$ takes all values with at most one exception infinitely often.  Hence~$g$ takes some small value $\zeta$ infinitely often and so, by the argument principle,  $\ind(\gamma_n,\zeta) \to \infty$ as $n\to\infty$. Thus $\ind(\gamma_n,0)\to\infty$ as $n\to\infty$ also, so~$\gamma_n$ winds round zero at least once for $n$ sufficiently large (see Figure~1).

\begin{figure}[hbt!]
\begin{center}
\def\svgwidth{1.2\textwidth}

\begingroup%
  \makeatletter%
  \providecommand\color[2][]{%
    \renewcommand\color[2][]{}%
  }%
  \providecommand\transparent[1]{%

    \renewcommand\transparent[1]{}%
  }%
  \providecommand\rotatebox[2]{#2}%
  \ifx\svgwidth\undefined%
    \setlength{\unitlength}{5413.08910476bp}%
    \ifx\svgscale\undefined%
      \relax%
    \else%
      \setlength{\unitlength}{\unitlength * \real{\svgscale}}%
    \fi%
  \else%
    \setlength{\unitlength}{\svgwidth}%
  \fi%
  \global\let\svgwidth\undefined%
  \global\let\svgscale\undefined%
  \makeatother%
  \begin{picture}(1,0.50715027)%
    \put(0,0){\includegraphics[width=\unitlength]{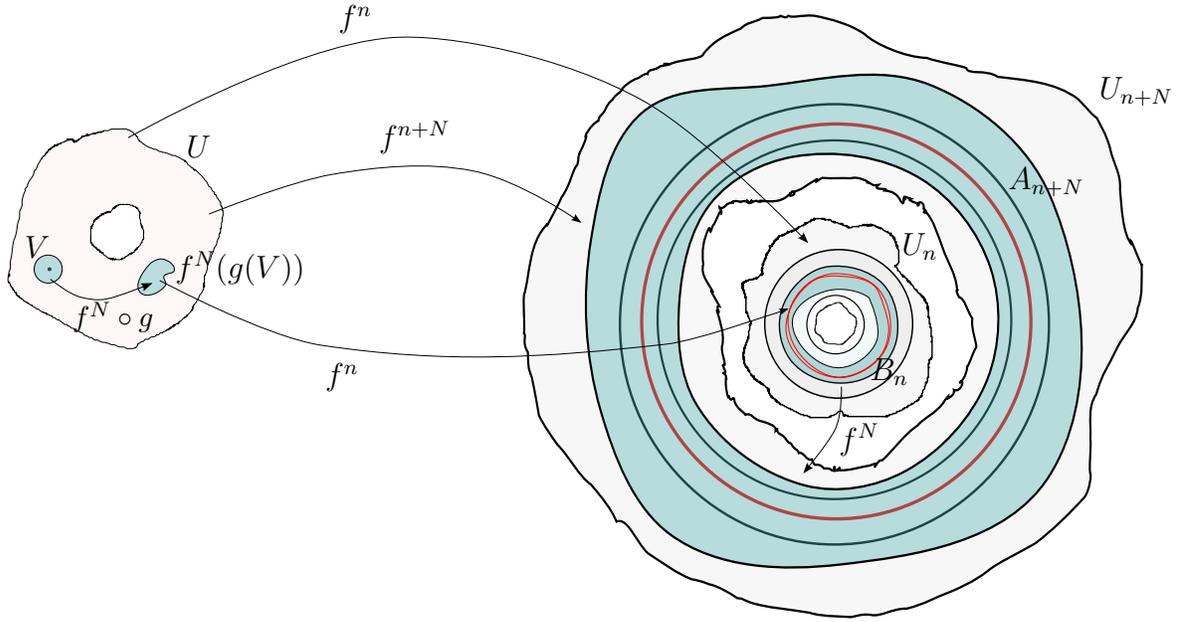}}%
    \put(0.13763298,0.35481415){\color[rgb]{0,0,0}\makebox(0,0)[lb]{\smash{$U$}}}%
    \put(0.24452879,0.17946202){\color[rgb]{0,0,0}\makebox(0,0)[lb]{\smash{$f^n$}}}%
    \put(0.28590004,0.36281605){\color[rgb]{0,0,0}\makebox(0,0)[lb]{\smash{$f^{n+N}$}}}%
    \put(0.84,0.4){\color[rgb]{0,0,0}\makebox(0,0)[lb]{\smash{$U_{n+N}$}}}%
    \put(0.77,0.33){\color[rgb]{0,0,0}\makebox(0,0)[lb]{\smash{$A_{n+N}$}}}%
    \put(0.01345213,0.27739908){\color[rgb]{0,0,0}\makebox(0,0)[lb]{\smash{$V$}}}%
    \put(0.05,0.22){\color[rgb]{0,0,0}\makebox(0,0)[lb]{\small{$f^N \circ g$}}}%
    \put(0.12930144,0.26247005){\color[rgb]{0,0,0}\makebox(0,0)[lb]{\smash{$f^N(g(V))$}}}%
    \put(0.688,0.28){\color[rgb]{0,0,0}\makebox(0,0)[lb]{\smash{$U_n$}}}%
    \put(0.66356604,0.18323484){\color[rgb]{0,0,0}\makebox(0,0)[lb]{\smash{$B_n$}}}%
    \put(0.64,0.13){\color[rgb]{0,0,0}\makebox(0,0)[lb]{\smash{$f^N$}}}%
    \put(0.25469035,0.45504275){\color[rgb]{0,0,0}\makebox(0,0)[lb]{\smash{$f^n$}}}%
  \end{picture}%
\endgroup%

\end{center}
\caption{\small A sketch of the wandering domain $U$ and its images $U_{n+N}$ and $U_n$. Images of $V$ are shaded in light blue, while $C_{n+N}$ and $\gamma_n$ are drawn in red.}
\label{MCWDFigure}
\end{figure}

Now let $\dnu$ be the hyperbolic length of $\gamma_n$ in $g(\AnN)$ and let $\dnd$ be the hyperbolic length of $\CnN$ in $\AnN$. By \eqref{Eq2} and Lemma~\ref{ann-hyp}, we have
\[
\dnd=\frac{\pi^2}{\alpha\log \rnN}.
\]
Since $g:\AnN\to g(\AnN)$ is holomorphic, and $\gamma_n\subset g(\CnN)$, we deduce by the contracting property of the hyperbolic metric that
\begin{equation}\label{Eq4}
\dnu\leq\dnd=\frac{\pi^2}{\alpha\log \rnN}.
\end{equation}
Next, by \eqref{Eq1}, we can choose a compact subset $K$ of $U$ containing both $z_0$ and $f^N(g(V))$. By~\eqref{Eq3}, we have
\[
\gamma_n\subset g(A_{n+N})\subset g(f^{n+N}(V))= f^{n}(f^N(g(V)))\subset f^n(K).
\]
By Theorem~\ref{BRS2}, there exist $N_2\in\N$ and $a_n,b_n$, for $n\ge N_2$, such that, for $n\ge N_2$,
\[
0< a_n\leq 1-\alpha\leq 1+\alpha\leq b_n<\infty
\]
and
\begin{equation}\label{Bn}
f^n(K)\subset B_n:=A({r_n}^{a_n},{r_n}^{b_n})\subset U_n.
\end{equation}
Since
\[
g(\AnN)\subset f^n(f^N(g(V)))\subset f^n(K)\subset B_n,
\]
we deduce, by the comparison principle for the hyperbolic metric, that  $\dnu$ is greater than or equal to the hyperbolic length of $\gamma_n$ in $B_n$.
Hence, using the fact that $\gamma_n$ winds at least once around $0$, we deduce by Lemma~\ref{ann-hyp} that
\[
\dnu\geq\frac{2\pi^2}{(b_n-a_n)\log r_n}>\frac{2\pi^2}{b_n\log r_n}.
\]
Hence, by~\eqref{Eq4},
\[
\frac{2\pi^2}{b_n\log r_n} < \frac{\pi^2}{\alpha\log \rnN}
\]
and so
\begin{equation}\label{rnN}
\log \rnN < \frac{1}{2\alpha}\log r_n^{b_n}.
\end{equation}
For sufficiently large $n$, $U_{n+1}$ surrounds $U_n$, by Theorem~\ref{Baker1}. Since
\[
\{z\in\C: |z|=r_n^{b_n} \}\subset \ov{B_n}\subset\ov{U_n},
\]
by~\eqref{Bn}, and
\[
\{z\in\C: |z|=\rnu \}\subset U_{n+1},
\]
we have $r_n^{b_n}<\rnu$. Hence, by~\eqref{rnN},
\begin{equation}\label{Eq5}
\log \rnN < \frac{1}{2\alpha}\log\rnu.
\end{equation}
By Corollary~\ref{BRS-Cor}, with $m=1$, and the fact that $N\ge 2$, we deduce that, for sufficiently large~$n$,
\begin{equation}\label{doubt}
\log\rnN \geq \frac{1}{2} \log M(r_{n+N-1},f)\geq \frac{1}{2} \log M(\rnu,f).
\end{equation}

Thus it follows from~\eqref{Eq5} that, for sufficiently large~$n$,
\[
\log M(\rnu,f) < \frac{1}{\alpha}\log\rnu\,,
\]
which contradicts \eqref{logMr}. Hence $g(U)\subset F(f)$.

Finally we claim that $g(U)$ must be contained in a multiply connected wandering domain of~$f$. Indeed, if
\[
g(U)\subset V,\quad\text{where } V \text{ is a simply connected Fatou component of } f,
\]
then, for $n\in \N$,
\begin{equation}\label{Vn}
g(f^n(U))=f^n(g(U))\subset f^n(V)\subset V_n,
\end{equation}
where $V_n$ is also a simply connected Fatou component of $f$; see \cite[Lemma~4.2]{RS3}, for example.

But we know from Theorem~\ref{BRS1} that, for sufficiently large $n$, the set $f^n(U)$ contains an open annulus of the form $A(R_n,2R_n)$, where $R_n\to \infty$ as $n\to\infty$. For $n$ sufficiently large and $r\in (R_n, 2R_n)$, the image of the circle $\{z:|z|=r\}$ under the function~$g$ winds at least once round a value close to~$0$, by Picard's theorem, and contains points of modulus $M(r,g)>r$, and also lies in $V_n$, by \eqref{Vn}. Since $V_n$ is a simply connected Fatou component of~$f$, it must meet $A(R_n,2R_n)$ and hence $f^n(U)$, and this is impossible.

Therefore $g(U) \subset V$, where~$V$ is a multiply connected wandering domain of~$f$. This completes the proof of Proposition~\ref{MCWD in Fatou Set}.
\end{proof}

To end the paper we point out an intriguing consequence of Theorem~\ref{main1}.

\begin{corollary}[Shared wandering domains]\label{shared}
Let $f$ and $g$ be permutable {\tef}s, and suppose that $U$ is a multiply connected wandering domain of $f$. Then there exists $N\in \N$ such that, for all $n\ge N$,
\[
f^n(U)\text{ is a multiply connected wandering domain of both }f \text{ and } g.
\]
In particular, $J(g) \ne \C$.
\end{corollary}
\begin{proof}
For each $n\in \N$, we know that $f^n(U)$ is a multiply connected wandering domain of~$f$, and also that $f^n(U)\subset V_n$, where $V_n$ is a Fatou component of~$g$, by Theorem~\ref{main1}. We claim that there must exist $N\in\N$ such that $V_n$ is multiply connected for $n\ge N$; for otherwise there is a subsequence $(V_{n_j})$ of these Fatou components that are all simply connected and have the properties that
\[
0\in \bigcap_{j=1}^{\infty} V_{n_j}\quad\text{and}\quad \bigcup_{j=1}^{\infty} V_{n_j} = \C,
\]
by Theorem~\ref{Baker1}, which is impossible.

For $n\ge N$ we have $V_n\subset \tilde{U}_n$ where $\tilde U_n$ is a Fatou component of~$f$, by Theorem~\ref{main1} again. Thus, for $n\ge N$, we have $\tilde U_n=f^n(U)$, so $f^n(U)=V_n$ is a Fatou component of~$g$ also.
\end{proof}

The conclusion of Corollary~\ref{shared} is that the transcendental entire functions~$f$ and~$g$ have infinitely many Fatou components in common. In this situation it is tempting to conjecture that their Julia sets must be identical.

Note that if we could deduce in Corollary~\ref{shared} that all the preimages of~$U$ under~$f$ were multiply connected wandering domains of~$g$, then it would follow by Lemma~\ref{bup} that $J(f)\subset J(g)$ and so, by applying a similar argument to preimages under~$g$, that $J(f)=J(g)$.

\end{document}